\documentclass{amsart}

\usepackage{latexsym}
\usepackage{amsmath}
\usepackage{amssymb}

\usepackage{multicol}

\usepackage[cp1251]{inputenc}
\usepackage[english]{babel}

\theoremstyle{plain}

\newtheorem{theorem}{Theorem} %[section]
\newtheorem{proposition}{Proposition} %[section]
\newtheorem{lem}{Lemma} %[section]
\newtheorem{cor}{Corollary} %[section]

\theoremstyle{definition}
\newtheorem{df}{Definition} %[section]
\newtheorem{rem}{Remark} %[section]

\newcommand{\sph}{S}
\newcommand{\cn}{\mathbb{C}^n}
\newcommand{\cpn}{S_n}
\newcommand{\meal}{m}
\newcommand{\si}{\sigma}
\newcommand{\sgh}{{\sigma}_n}

\newcommand{\Dbb}{\mathbb{D}}
\newcommand{\Tbb}{\mathbb{T}}
\newcommand{\Nbb}{\mathbb{N}}

\newcommand{\zz}{\mathbb{Z}_+^2}

\newcommand{\spn}{S_n}

\newcommand{\za}{\zeta}
\newcommand{\de}{\delta}

\newcommand{\la}{\lambda}
\newcommand{\al}{\alpha}
\newcommand{\riz}{\Pi}

\newcommand{\dg}{\textrm{deg}\,}
\newcommand{\spec}{\textrm{spec}\,}

\numberwithin{equation}{section}

\begin{document}

\date{}

\author{Evgueni Doubtsov}
\address{St.~Petersburg Department
of Steklov Mathematical Institute, Fontanka 27, St.~Petersburg 191023, Russia}
\email{dubtsov@pdmi.ras.ru}

\title[Mutual singularity of Riesz products]{Mutual singularity of Riesz products on the unit sphere}

\begin{abstract}
We prove analogs of Peyri\`ere's mutual singularity theorem
for standard and generalized Riesz products on the unit sphere of $\mathbb{C}^n$, $n\ge 2$.
As a corollary, we obtain an analog of Zygmund's dichotomy for the Riesz products under consideration.

\end{abstract}


\thanks{This research was supported by the Russian Science Foundation (grant No.~23-11-00171),
https://rscf.ru/project/23-11-00171/}

\maketitle

\section{Introduction}\label{s_int}

In the present paper, we consider Riesz product measures and related objects
on the unit sphere $S = \spn = \{\za\in\cn : |\za| = 1\}$, $n \ge 2$.
So we start by recalling the relevant definitions of analysis on the sphere $\spn$.

\subsection{Basics of harmonic analysis on the unit sphere of $\cn$}

Let $\mathcal{U}(n)$ denote the group of unitary operators on the Hilbert space $\cn$, $n\ge 2$.
Observe that $\spn = \mathcal{U}(n)/\mathcal{U}(n-1)$, hence, $\spn$ is a homogeneous space.
General constructions of abstract harmonic analysis are explicitly implemented
on the sphere $\spn$ in terms of the spaces $H(p,q)$, $(p,q)\in\zz$.

\begin{df}
Fix a dimension $n$, $n\ge 2$.
Let $H(p, q)= H(p,q; n)$ denote the space of all homogeneous harmonic polynomials
of bidegree $(p, q) \in \zz$.
By definition, this means that the polynomials under consideration have degree
$p$ in variables $z_1, z_2, \dots, z_n$, degree $q$ in variables
$\overline{z}_1, \overline{z}_2, \dots, \overline{z}_n$,
and have total degree $p+q$.

For the restriction of $H(p, q)$ on $S$, one uses the same symbol.
The elements of $H(p, q)$ are often called \textit{complex} spherical harmonics.
\end{df}

Let $\si = \sgh$ denote the normalized Lebesgue measure on the unit sphere. Observe that
\[
L^2(\si) = \operatornamewithlimits{\oplus}_{(p,q)\in \zz} H(p, q).
\]

Specific aspects of the harmonic analysis on $S$ are illustrated by the following
multiplication rule for the spaces $H(p, q)$: if $f \in H(p, q)$ and
$g \in H(r, s)$, then
\[ fg\in
\sum_{\ell=0}^L
H(p + r -\ell, q + s -\ell),
\]
where $L = \min(p, s) + \min(q, r)$.
See \cite[Chapter~12]{Ru80} for the proofs of the above facts and
further results about the complex spherical harmonics.

Let $M(\spn)$ denote the space of complex Borel measures on the unit sphere $\spn$.
Let $K_{p,q}(z, \za)$ denote the reproducing kernel for the Hilbert space
$H(p, q) \subset L^2(\sph)$. The polynomial
\[
\mu_{p,q}(z) = \int_S
K_{p,q} (z, \za)\, d\mu(\za), \quad z\in S,
\]
is called the $H(p, q)$-projection of $\mu\in M(S)$.
Let $\spec(\mu)$ denote the spectrum of $\mu\in M(S)$
in terms of the spherical harmonics. Namely, by definition,
\[
\spec(\mu) =
\left\{
(p, q) \in \zz:\, \mu_{p,q} \neq \mathbf{0}
\right\}.
\]

\subsection{Singular Riesz products on the sphere}

The classical Riesz product measures are defined on the unit circle $\Tbb= \{\za\in\mathbb{C}: |\za|=1\}$.
%see Sect.~\ref{ss_clas_riz}.
While there is no canonical generalization of the Riesz product construction to $\spn$,
useful variants of Riesz products are based on Ryll--Wojtaszczyk polynomials
(see, e.g., \cite{DouAIF})
and are motivated by Aleksandrov's abstract inner function construction \cite{Aab84};
see also \cite{Bou85}, where similar product measures
are based on a bounded orthonormal basis in the Hardy space $H^2(\spn)$.

Let $\Dbb$ denote the open unit disk of $\mathbb{C}$.
Given a lacunary sequence $\{j_k\}_{k=1}^\infty \subset \Nbb$ and
a sequence $\{a_k\}_{k=1}^\infty \subset \Dbb$,
the (standard) Riesz product $\riz(R, J, a)$ is defined by the formal equality
\begin{equation}\label{e_standard_df}
\riz(R, J, a) = \prod_{k=1}^\infty
\left(
\frac{\overline{a}_k \overline{R}_{j_k}}{2}
+ 1 +
\frac{a_k R_{j_k}}{2}
\right),
\end{equation}
where $\{R_j\}_{j=1}^\infty$ is an appropriate sequence of holomorphic homogeneous polynomials.
By definition, $(R, J, a)$ is called a Riesz triple; see Sect.~\ref{ss_clas_riz} for details.
The standard Riesz product construction is used
in \cite{DouAIF}, it serves as an intermediate step to more sophisticated pluriharmonic Riesz
product construction, where the set $J$ is constructed by induction with large lacunae.
Hence, it is worth mentioning that the index set $J$ is predefined in \eqref{e_standard_df}.

In this paper, we obtain, in particular, the following
variant of Peyri\`ere's mutual singularity criterion \cite{Pe75} for
two Riesz products on the unit sphere.

\begin{theorem}\label{t_snglr_intro}
Let $(R, J, a)$ and $(R, J, b)$ be Riesz triples on the sphere $S_n$, $n\ge 2$.
Assume that $a-b \notin \ell^2$.
Then there exists a sequence $U = \{U_j\}_{j=1}^\infty$, $U_j\in \mathcal{U}(n)$,
such that
\[
\riz(R\circ U, J, a)\ \textrm{and}\ \ \riz(R\circ U, J, b)\ \textrm{are mutually singular}.
\]
\end{theorem}

Also, we obtain a similar theorem about mutually singular generalized Riesz products.

\subsection{Organization of the paper} Definitions and auxiliary facts about
Riesz products are presented in Section~\ref{s_riz}.
In particular, we show that a Riesz product decomposes as the integral of its formal slice-products;
see Lemma~\ref{l_disint}. Analogous decomposition for the variation
of $\riz(R, J, a)- \riz(R, J, b)$ is obtained in Proposition~\ref{p_disint_sng_two}.
Theorem~\ref{t_snglr_intro} is proved in Section~\ref{s_zyg_sph}.
Generalized Riesz products are studied in the final Section~\ref{s_riz_G}.
In particular, we obtain a direct analog of the classical Zygmund theorem; see Corollary~\ref{c_zyg_true}.

\section{Riesz products}\label{s_riz}

\subsection{Classical Riesz products}\label{ss_clas_riz}

By definition, $(J, a)$ is an admissible pair if
$J=\{j_k\}_{k=1}^\infty$, $j_{k+1}/j_k \ge 3$, and
$a=\{a_k\}_{k=1}^\infty \subset \Dbb$.
Recall that a classical Riesz product $\mu$ based on an admissible pair $(J, a)$
is defined on the unit circle $\Tbb$ by
\[
\mu = \mu(J, a) := \prod_{k=1}^\infty \left(
\frac{\overline{a}_k \overline{z}^{j_k}}{2} + 1 +
\frac{{a}_k {z}^{j_k}}{2}
\right),
\quad z \in\Tbb,
\]
where the infinite product is understood in the weak*-sense.
Each partial product is positive, thus, the assumption $j_{k+1}/j_k \ge 3$ guarantees the convergence of the product.

Let $\meal$ denote the normalized Lebesgue measure on the unit circle $\Tbb$.
Zygmund’s theorem \cite{Zy} establishes the following dichotomy:
\begin{itemize}
  \item [(i)] if $\sum_{k=1}^\infty |a_k|^2 < \infty$, then $\mu$ is absolutely continuous with respect to $\meal$
  (in brief, $\mu\ll \meal$) and $d\mu/d\meal \in L^2(\Tbb)$;
  \item [(ii)] if $\sum_{k=1}^\infty |a_k|^2 = \infty$, then $\mu$ is singular with respect to $\meal$
  (in brief, $\mu\bot\meal$).
\end{itemize}

Clearly, part (ii) of Zygmund's dichotomy is a particular case of the following
result about mutual singularity of Riesz products.

\begin{proposition}[J.~Peyri\`ere \cite{Pe75}]\label{p_pey}
Let $(J,a)$ and $(J, b)$ be admissible pairs.
Assume that $a-b\notin \ell^2$. Then
\[
\mu(J, a) \bot (J, b).
\]
\end{proposition}

The main results of the present paper extend the above singularity criterion
to the standard and generalized Riesz products on the unit sphere.

\subsection{Riesz products on the sphere}\label{ss_sph_riz}
Part (i) of Zygmund's dichotomy indicates that non-trivial (singular) examples
of analogs of Riesz products on the sphere could be based
on the homogeneous holomorphic polynomials introduced by
Ryll and Wojtaszczyk \cite{RW83}
or on holomorphic polynomials with similar properties;
see, e.g., \cite{DouAIF} and \cite{Du03}.

%\begin{df}\label{d_RW}
We say that $\{R_j\}_{j=1}^\infty$
 is a Ryll--Wojtaszczyk sequence (in brief, RW-sequence) with a constant $\delta \in (0, 1)$ if
\begin{itemize}
\item $R_j \in H(j, 0)$, i.e., $R_j$ is a homogeneous holomorphic polynomial of degree~$j$,
\item $\|R_j\|_{L^\infty(\sph)} = 1$,
\item $\|R_j\|_{L^2(S)} \ge \delta$ for all $j =1,2,\dots$.
\end{itemize}
%\end{df}

\begin{df}\label{d_Rpair}
Let $R = \{R_j\}_{j=1}^\infty$ be a Ryll--Wojtaszczyk sequence,
$J = \{j_k\}_{k=1}^\infty\subset \Nbb$, $j_{k+1}/j_k \ge 3$,
and $a = \{a_k\}_{k=1}^\infty \subset \Dbb$.
Then $(R,J,a)$ is called a Riesz triple.
\end{df}

Each Riesz triple generates a (standard) Riesz product (see \cite{DouAIF}).
Namely, the standard Riesz product
$\riz(R, J, a)$ is defined by the formal equality
\[
\riz(R, J, a) = \prod_{k=1}^\infty
\left(
\frac{\overline{a}_k \overline{R}_{j_k}}{2}
+ 1 +
\frac{a_k R_{j_k}}{2}
\right).
\]
To prove the convergence of the above product, consider the partial products
\begin{equation}\label{e_Riz_dPartial}
\riz_\varkappa(R, J, a) := \prod_{k=1}^\varkappa
\left(
\frac{\overline{a}_k \overline{R}_{j_k}}{2}
+ 1 +
\frac{a_k R_{j_k}}{2}
\right).
\end{equation}
Fix a polynomial $\mathcal{P}$ on $\sph$.
Since $j_{k+1}/j_k \ge 3$, we have
\[
\spec (\riz_{\varkappa+\ell} - \riz_\varkappa) \cap \spec \mathcal{P} = \varnothing\quad\textrm{for all}\ \ell \in \Nbb
\]
provided that $\varkappa$ is sufficiently large.
Also, we have $\riz_\varkappa \ge 0$ and $\|\riz_\varkappa\|_{L^1(\sph)} = 1$.
Therefore, the partial products
$\riz_\varkappa(R, J, a)$
converge weakly* as $\varkappa\to \infty$ to a probability measure.
So use the above symbol $\riz(R, J, a)$ for this limit.

Now, fix a $\za\in \sph$.
Clearly, the slice product
\[
\riz_\za(R, J, a)(\la) := \riz(R(\la\za), J, a),\quad \la\in \Tbb,
\]
is the classical Riesz
product based on the admissible pair
$\left(\{a_k R_{j_k} (\za)\}_{k=1}^\infty, \{j_k\}_{k=1}^\infty \right)$.
In particular, $\riz_\za(R, J, a)$ is a correctly defined probability measure on the unit circle $\Tbb$.
This observation often reduces a problem about standard Riesz products
to a question about classical Riesz products on the unit circle.
In fact, the main technical idea is
to decompose a standard Riesz product as the integral of its slice-products.

\subsection{Slices and decompositions}

\begin{lem}\label{l_disint}
Let $(R, J, a)$ be a Riesz triple on the sphere $S_n$, $n\ge 2$,
and let $\riz(R, J, a)$ denote the corresponding Riesz product. Then
\begin{equation}\label{e_disint_main}
 \riz(R, J, a) = \int_{\cpn} \riz_\za(R, J, a)\, d\sgh(\za)
\end{equation}
in the following weak sense:
\begin{equation}\label{e_disint_CS}
 \int_{S_n} f\, d \riz(R, J, a) = \int_{\cpn} \int_{\Tbb} f\, d\riz_\za(R, J, a)\, d\sgh(\za)
\end{equation}
for all $f\in C(S_n)$.
\end{lem}
\begin{proof}
Let $\riz_\varkappa:= \riz_\varkappa(R, J, a)$ be defined by \eqref{e_Riz_dPartial}.
Since $\riz_\varkappa \in C(\spn)$, the integration by slices formula (see, e.g., \cite{Ru80}) guarantees that
\begin{equation}\label{e_riz_r_decomp}
\int_{S_n} f\, d\riz_\varkappa = \int_{\cpn} \int_{\Tbb} f\, d(\riz_\varkappa)_\za \, d\sgh(\za),\quad  f\in C(\spn),
\end{equation}
where $(\riz_\varkappa)_\za$ is considered as a measure on $\Tbb$.

By the hypothesis, the Riesz product $\riz(R, J, a)$ is correctly defined, hence,
\begin{equation}\label{e_LH_riz}
  \riz_\varkappa \overset{w^\ast}\longrightarrow \riz(R, J, a).
\end{equation}
Next, $(\riz_\varkappa)_\za =(\riz_\za)_\varkappa$ is a partial product of the classical Riesz product $\riz_\za$,
in particular, $(\riz_\za)_\varkappa \overset{w^\ast}\longrightarrow \riz_\za$.
Therefore,
%using the integration by slices formula (see \cite{Ru80}), we obtain
\begin{equation}\label{e_RH_slices}
\int_{\cpn} \int_{\Tbb} f\, d(\riz_\za)_\varkappa\, d\sgh(\za) \to
\int_{\cpn} \int_{\Tbb} f\, d\riz_\za \, d\sgh(\za)\quad \textrm{as}\ r\to\infty.
\end{equation}
Combining properties \eqref{e_riz_r_decomp}, \eqref{e_LH_riz} and \eqref{e_RH_slices},
we obtain \eqref{e_disint_CS}, as required.
\end{proof}

\begin{proposition}\label{p_disint_sng_two}
Let $(R, J, a)$ and $(R, J, b)$  be Riesz triples on the sphere $S_n$, $n\ge 2$,
and let $\riz(R, J, a)$, $\riz(R, J, b)$ denote the corresponding Riesz products.
Assume that
\[
\riz_\za(R, J, a) \bot \riz_\za(R, J, b)
\]
for $\sgh$-almost all $\za\in\cpn$.
Then
\[
\riz(R, J, a)\bot\riz(R, J, b).
\]
\end{proposition}
\begin{proof}
Since $\riz_\za(R, J, a)$ and $\riz_\za(R, J, b)$, $\za\in\cpn$, are probability measures,
we clearly have
\[
\int_{\cpn} \|\riz_\za(R, J, a) - \riz_\za(R, J, b)\| \,d\sgh(\za) < \infty.
\]
The above property and the decomposition formula \eqref{e_disint_main} for
the products $\riz(R, J, a)$ and $\riz(R, J, b)$
guarantee that $\riz(R, J, a) - \riz(R, J, b)$ is a decomposable measure in the sense of \cite[Definition~2]{AD20}.
Hence, by Theorem~2.6 from \cite{AD20},
\begin{equation}\label{e_var2_slices}
|\riz(R, J, a) - \riz(R, J, b)| = \int_{\cpn} |\riz_\za (R, J, a) - \riz_\za (R, J, a)|\, d\sgh(\za)
\end{equation}
in the weak sense.
Next, given two probability measures $\rho_1$ and $\rho_2$ on $\spn$, observe that
$\rho_1 \bot \rho_2$ if and only if $\|\rho_1 - \rho_2\|=2$.
Thus, property \eqref{e_var2_slices} and the hypothesis of the proposition
guarantee that $\riz(R, J, a)\bot\riz(R, J, b)$, as required.
\end{proof}

\section{Mutually singular Riesz products on the unit sphere}\label{s_zyg_sph}

Let $U = \{U_j\}_{j=1}^\infty$ be a sequence of unitary operators
on $\cn$ and let $R = \{R_j\}_{j=1}^\infty$ be a sequence of polynomials. 
By definition, put $R\circ U = \{R_j \circ U_j\}_{j=1}^\infty$.

\begin{proof}[Proof of Theorem~\ref{t_snglr_intro}]
Since $R$ is an RW-sequence and $a - b\notin \ell^2$, we have
\[
\sum_{k=1}^\infty |a_k - b_k|^2 \| R_{j_k}\|_{L^2(S)}^2 = \infty.
\]
Thus, by the scrambling lemma (see, e.g., \cite[Lemma~7.2.7]{Ru80}),
there exist unitary operators $U_k\in \mathcal{U}(n)$, $k=1, 2, \dots$, such that
\begin{equation}\label{e_Uslice_infty}
\sum_{k=1}^\infty |a_k - b_k|^2 |R_{j_k}\circ U_k (\za)|^2 = \infty.
\end{equation}
for $\si$-almost all $\za\in S$.
Fix a $\za\in S$ such that \eqref{e_Uslice_infty} holds,
put $\al_{k, \za} = a_k R_{j_k}\circ U_k (\za)$, $\beta_{k, \za} = b_k R_{j_k}\circ U_k (\za)$, $k=1,2,\dots$, $\al_\za=\{\al_{k, \za}\}_{k=1}^\infty$ and $\beta_\za=\{\beta_{k, \za}\}_{k=1}^\infty$.
Consider the classical Riesz products $\mu(J, \al_\za)$ and $\mu(J, \beta_\za)$.
Rewriting \eqref{e_Uslice_infty} as $\al_\za - \beta_\za\notin\ell^2$, we obtain
\begin{equation}\label{e_Uslice_snglr}
\mu(J, \al_\za) \bot \mu(J, b_\za).
\end{equation}
by Proposition~\ref{p_pey}.
Now, observe that $\mu(J, \al_\za)$ and $\mu(J, \beta_\za)$ coincide with the slice-products
$\riz_\za(R_{j_k}\circ U_k, J, a_k)$ and $\riz_\za(R_{j_k}\circ U_k, J, b_k)$, respectively.
Therefore, by decomposition formula,
\[
\begin{split}
\riz(R_{j_k}\circ U_k, J, a_k)
&= \int_{\cpn} \mu(J, \al_\za)\, d\sgh(\za), \\
\riz(R_{j_k}\circ U_k, J, b_k)
&= \int_{\cpn} \mu(J, \beta_\za)\, d\sgh(\za).
\end{split}
\]
Hence, by \eqref{e_Uslice_snglr} and Proposition~\ref{p_disint_sng_two},
\[
\riz(R\circ U, J, a)\bot\si,
\]
as required.
\end{proof}

Putting $b=\mathbf{0}$ in Theorem~\ref{t_snglr_intro}, we obtain the following analog of part~(ii)
of Zygmund's dichotomy.

\begin{cor}\label{c_zyg_sph}
Let $(R, J, a)$ be a Riesz triple on the sphere $S_n$, $n\ge 2$.
Assume that $a \notin \ell^2$.
Then there exists a sequence $U = \{U_j\}_{j=1}^\infty$, $U_j\in \mathcal{U}(n)$,
such that
\[
\riz(R\circ U, J, a)\bot \si.
\]
\end{cor}

\begin{rem}
Let $(R, J, a)$ be a Riesz triple on the sphere $S_n$, $n\ge 2$.
Assume that $a \in \ell^2$. Applying standard estimates
(see, e.g., \cite[Theorem~4.1(i)]{DouAIF}),
we conclude that  $\riz(R, J, a)\ll\si$.
In other words, part~(i) of Zygmund's dichotomy
extends to the standard Riesz products.
\end{rem}

\begin{rem}\label{r_U_isness}
Since the definition of an RW-sequence restricts only the $L^2$-norms
of the polynomials $R_j$, $j=1,2, \dots$,
it is not difficult to give an example of an RW-sequence $R$ such that
an auxiliary sequence $U$ is necessary in Corollary~\ref{c_zyg_sph}, that is,
the property $\riz(R, J, a)\bot\si$ does not hold.
\end{rem}

\section{Mutually singular generalized Riesz products}\label{s_riz_G}

\subsection{Generalized Riesz pairs}
This section is motivated, in particular, by Remark~\ref{r_U_isness}:
it is desirable to change the standard Riesz product construction
in such a way that a direct analog of part (ii) of Zygmund's dichotomy
holds for the corresponding Riesz type products.
To realize this program, we replace the RW-polynomials by certain non-homogeneous
holomorphic polynomials with special properties.
In the setting of the unit circle, the corresponding Riesz type product measure
is called a generalized Riesz product.
To obtain analogs of generalized Riesz products on the sphere, we use the following
holomorphic polynomials with arbitrarily large lacunae in spectrum.

\begin{df}\label{d_RW_Li}
Assume that, for all $j, L \in \Nbb$, we are given
$D(j, L)\in \Nbb$ and homogeneous holomorphic polynomials
$ W_\varkappa = W_\varkappa(j, L)$, $\varkappa = 1, 2, . . . ,D(j, L)$, such that
\begin{itemize}
  \item $\dg W_\varkappa \ge j$;
  \item $|\dg W_{\varkappa_1} - \dg W_{\varkappa_2}| \ge L$ (lacunaes in spectrum);
  \item $\left|\sum_\varkappa W_\varkappa(\za)\right| \le 1$ for all $\za\in S$;
  \item $\sum_\varkappa |W_\varkappa(\za)| \ge \de$ for all $\za\in\sph$ and a universal constant $\de > 0$.
\end{itemize}
Then put
\[
R(j, L) = \sum_\varkappa W_\varkappa(j, L), \quad j, L\in \Nbb,
\]
and define 
\[
R = \{R(j, L)\}_{j,L}.
\]
Also, fix a coefficient sequence $a = \{a_k\}_{k=1}^\infty \subset \Dbb$.
By definition, $(R, a)$ is called a generalized Riesz pair.
\end{df}

It is worth mentioning that the pairs from the above definition do exist, various examples and 
further applications in the pluriharmonic setting are given in \cite{Du03},
where such pairs are called $L^\infty$-generalized.

\subsection{Generalized Riesz products}\label{ss_GRP}
Fix a generalized Riesz pair $(R, a)$.
The index set $J = \{j_k\}_{k=1}^\infty$ is constructed by induction.

Firstly, fix a $j_1 \in \Nbb$ and define 
\[
\riz_1 = 1 + \mathrm{Re}\, [a_1 R(j_1, 1)].
\]

Step $k + 1$. By the induction hypothesis, we are given a polynomial $\riz_k$.
Put $L_{k+1} = 2 \dg \riz_k + 2$ and define
\[
\riz_{k+1} = \riz_k\cdot (1 + \mathrm{Re}\, [a_{k+1} R(j_{k+1}, L_{k+1})]).
\]
The index $j_{k+1}$ is selected so large that the set
$\spec(\riz_{k+1} - \riz_k)$  does not intersect a square 
$[0,M]^2 \subset \zz$ such that $[0,M]^2 \supset \spec(\riz_k)$.

The above properties guarantee that $\riz_k$ weakly* converge in $M(\sph)$.
The limit probability measure $\riz = \riz(R, J, a)$ is called
a generalized Riesz product.

\subsection{Generalized Riesz products on the unit circle}
Definition~\ref{d_RW_Li} is applicable to polynomials on $\Tbb$.
Observe that the corresponding polynomials $W_\varkappa = W_\varkappa(j, L)$ are holomorphic monomials.
In this setting,
the construction described in Subsection~\ref{ss_GRP} is that of a usual generalized Riesz product.

\begin{proposition}\label{p_gnrlz_1}
Suppose that $\riz(R, J, a)$ and $\riz(R, J, b)$ are correctly defined gene\-ra\-lized Riesz products 
on $\Tbb$.
Assume that $a-b\notin \ell^2$. Then
\[
\riz(R, J, a) \bot \riz(R, J, b).
\]
\end{proposition}
\begin{proof}
It is known that Peyri\`ere's proof of Proposition~\ref{p_pey}
extends to the pairs of generalized Riesz products on the unit circle 
(see, for example, \cite[Theorem~3.3]{Du03}).
\end{proof}

\subsection{Mutually singular generalized Riesz products}

\begin{proposition}\label{p_gnrlz}
Suppose that $\riz(R, J, a)$ and $\riz(R, J, b)$
are correctly defined generalized Riesz products
on the sphere $S_n$, $n\ge 2$.
Assume that $a-b \notin \ell^2$.
Then
\[
\riz_\za(R, J, a)\bot \riz_\za(R, J, b).
\]
for all $\za\in\sph$.
\end{proposition}
\begin{proof}
Fix a point $\za\in \sph$.
Recall that $\riz_\za(R, J, a)(\la) = \riz(R(\la\za), J, a)$, $\la\in\Tbb$.
Therefore, $\riz_\za(R, J, a)$ and $\riz_\za(R, J, b)$ are generalized Riesz products on $\Tbb$ and are based on the
sequences $a$ and $b$, respectively. Since $a-b\notin \ell^2$, we have 
$\riz_\za(R, J, a)\bot \riz_\za(R, J, b)$ by Proposition~\ref{p_gnrlz_1}.
\end{proof}

\begin{theorem}\label{t_msnglr_generlz}
Suppose that $\riz(R, J, a)$ and $\riz(R, J, b)$
are correctly defined generalized Riesz products on the sphere $S_n$, $n\ge 2$.
Assume that $a-b \notin \ell^2$.
Then
\[
\riz(R, J, a) \bot \riz(R, J, b).
\]
\end{theorem}
\begin{proof}
Indeed, by disintegration formula,
\[
\begin{split}
\riz(R_{j_k}\circ U_k, J, a_k) &= \int_{\cpn} \riz_\za(R_{j_k}\circ U_k, J, a_k)\, d\sgh(\za),\\
\riz(R_{j_k}\circ U_k, J, b_k) &= \int_{\cpn} \riz_\za(R_{j_k}\circ U_k, J, b_k)\, d\sgh(\za).
\end{split}
\]
Hence, combining Proposition~\ref{p_gnrlz} and Lemma~\ref{l_disint}, we obtain
\[
\riz(R\circ U, J, a)\bot \riz(R\circ U, J, b),
\]
as required.
\end{proof}

Setting $b=\mathbf{0}$ in Theorem~\ref{t_msnglr_generlz}, we obtain the following direct analog
of part~(ii) of Zygmund's dichotomy.

\begin{cor}\label{c_zyg_true}
Suppose that $\riz(R, J, a)$ is a correctly defined generalized Riesz product
on the sphere $S_n$, $n\ge 2$.
Assume that $a\notin \ell^2$.
Then
$
\riz(R, J, a) \bot \si.
$
\end{cor}

\bibliographystyle{amsplain}
\providecommand{\bysame}{\leavevmode\hbox to3em{\hrulefill}\thinspace}
\providecommand{\MR}{\relax\ifhmode\unskip\space\fi MR }
% \MRhref is called by the amsart/book/proc definition of \MR.
\providecommand{\MRhref}[2]{%
  \href{http://www.ams.org/mathscinet-getitem?mr=#1}{#2}
}
\providecommand{\href}[2]{#2}

\end{document}